\documentclass{article}
\usepackage[left=3cm,right=3cm,top=3.5cm,bottom=3.5cm]{geometry} 
\usepackage{amsmath} 
\usepackage{amsmath,amsthm}
\usepackage{amssymb,latexsym}
\usepackage{enumerate}

\DeclareMathOperator{\bin}{bin}
\DeclareMathOperator{\ord}{ord}

\newtheorem{thm}{Theorem}[section]
\newtheorem{col}[thm]{Corollary}
\newtheorem{lem}[thm]{Lemma}
\newtheorem{conj}[thm]{Conjecture}

\def\l({\left(}
\def\r){\right)}

\begin{document}

\title{A note on the arithmetic properties of Stern Polynomials}
\author{Maciej Gawron}
\date{\today}
\maketitle
\footnote{2010 \emph{Mathematics Subject Classification}: 11B83.}
\footnote{This research was supported by Grant DEC-2012/07/E/ST1/00185.}
\footnote{Key words and phrases. Stern polynomials, Stern diatomic sequence, Automatic Sequences.}

\begin{abstract}
We investigate the Stern polynomials defined by $B_0 ( t ) =0,B_1 ( t ) =1$, and for $n \geq 2$ by the recurrence relations 
$B_{2n}( t) =tB_{n}( t) ,$ $B_{2n+1}( t) =B_n( t) +B_{n+1}( t) $.
We prove that all possible rational roots of that polynomials are $0,-1,-1/2,-1/3$. We give complete characterization of $n$
such that $\deg( B_n)  = \deg( B_{n+1}) $ and $\deg( B_n) =\deg( B_{n+1}) =\deg( B_{n+2}) $. Moreover, we present 
some result concerning reciprocal Stern polynomials. 
\end{abstract}

\begin{section}{Introduction}
\indent  We consider the sequence of Stern polynomials $( B_n( t) ) _{n \geq 0}$ defined by the following formula

 \begin{equation*}
  B_0( t) =0, \quad B_1( t)  = 1, \quad B_n( t)  = \begin{cases} tB_{n/2}( t)  & \text{ for } n \text{ even,} \\
					B_{( n+1) /2}( t) +B_{( n-1) /2}( t)  & \text{ for } n \text{ odd.}
                                       \end{cases}
 \end{equation*} 
This sequence was introduced in \cite{KMP} as a polynomial analogue of the Stern  \cite{STR} sequence $(s_n) _{n \in \mathbb{N}}$, where 
$s_n = B_n(1)$.
Arithmetic properties of the Stern polynomials and the sequence of degrees of Stern polynomials $e( n)  = \deg B_n( t) $
were considered in \cite{UL1, UL2}. Reducibility properties of Stern polynomials were considered in $\cite{SCH}$.  
The aim of this paper is to resolve some open problems and conjectures given in \cite{UL1}. 

\indent In Section 2 we prove that the only possible rational roots of the Stern polynomials are in the set $\{0,-1, - \frac 12, - \frac 13 \}$.
We prove that each of those numbers is a root of infinitely many Stern polynomials. We also give some characterization of $n$ 
such that $B_n( t) $ has one of these roots. 

\indent It was proven in \cite{UL1} that there are no four consecutive
Stern polynomials with equal degrees. In Section 3 we characterize the set of $n$ for which $e( n) =e( n+1) $, i.e. 
these $n$ for which two consecutive Stern polynomials has the same degree. We 
also characterize the set of $n$ for which $e( n) =e( n+1) =e( n+2) $. 

\indent In Section 4 we investigate reciprocal Stern polynomials, i.e. polynomials such that $B_n( t) =t^{e( n) }B_n( \frac 1t) $. 
It was observed in \cite{UL1} that $B_n( t) $ is reciprocal for each $n$ of the form $2^m-1$ or $2^m-5$. We give two infinite sequences
$( u_n) ,( v_n) $ such that, $B_n( t) $ is reciprocal for each $n$ of the forms $2^m-u_n$, and $2^m-v_n$. 

\indent For $n \in \mathbb{N}$ we denote by $\bin( n) $ sequence of its binary digits without leading zeros. 
Let $w=( w_0,w_1,\ldots ,w_n) $ be a finite sequence of zeros and ones, we denote by $\overline{w}= \overline{w_0w_1\ldots w_n}$ the number 
$w_0\cdot 2^{n}+w_1 \cdot 2^{n-1}+\ldots+w_n$. For a finite sequence $w$ of zeros and ones,
we denote by $w^k$ concatenation of $k$ copies of $w$. 
For a set $A \subset \mathbb{Z}$ and an integer $n$, we denote by $A+n$ the set $\{a+n|a \in A\}$. For a subset $A \subset \mathbb{N}$
of non-negative integers we say that it has lower density $\delta$ if $\liminf\limits_{n \to \infty}  \frac{\#( A \cap [1,n]) }{n} = \delta$.
We say that it has upper density $\delta$ if $\limsup\limits_{n \to \infty}  \frac{\#( A \cap [1,n]) }{n} = \delta$. If lower and upper density
are both equal $\delta$ then we say that our set has density $\delta$.  
The symbol $\log$ means always binary logarithm. We use standard Landau symbols.

\end{section}

\begin{section}{Rational roots of Stern polynomials}
 We prove the following theorem, which can be found as a conjecture in \cite[Conjecture 6.1]{UL1}. 
 \begin{thm} If $a \in \mathbb{Q}$ and there exists a positive integer $n$ such that $B_n( a) =0$ then
 $a \in \{0,-1,-\frac{1}{2}, - \frac{1}{3} \}$.
 \end{thm}
 \begin{proof} Let $k \geq 4$ be an integer. We define a sequence $\{b_n\}_{n=0}^{\infty}$ by the formula $b_n = B_n( -\frac{1}{k}) $.
 We prove that the following inequality holds  
 \begin{equation}
   \label{eq1}
  b_{2n+1} > \frac{1}{2} \max \{|b_n|,|b_{n+1}|\} > 0
 \end{equation}
for every integer $n \geq 0$. In order to prove \eqref{eq1} we use induction. For $n=0$ we get $1=b_1 > \frac{1}{2} \max \{ |b_0|,|b_1| \} = \frac{1}{2}>0$. Suppose that the statement is true
for all $m<n$. Let us consider two cases $n=2s+1$ and $n=2s$. 

If $n=2s$ then we have
\begin{equation}
 b_{4s+1} = b_{2s}+b_{2s+1} = -\frac{1}{k} b_s + b_{2s+1} \geq b_{2s+1}-\frac{1}{4}|b_s| > 
 b_{2s+1}-\frac{1}{2} b_{2s+1} =  \frac{1}{2} b_{2s+1}.
\end{equation}
Moreover, $b_{2s+1} \geq \frac{1}{2} |b_s| $ and thus $\max \{|b_{2s}|,|b_{2s+1}|\} = \max \{\frac{1}{k}|b_{s}|,|b_{2s+1}|\} = b_{2s+1}>0$.  
So in this case our inequality is proved.

If $n=2s+1$ then we have
\begin{equation}
 b_{4s+3} = b_{2s+1}+b_{2s+2} = b_{2s+1}-\frac{1}{k}b_{s+1} \geq b_{2s+1}-\frac{1}{4}|b_{s+1}| > 
 b_{2s+1} - \frac{1}{2} b_{2s+1} = \frac{1}{2} b_{2s+1}.
\end{equation}
Moreover, $b_{2s+1} \geq \frac{1}{2} |b_{s+1}|$ and thus $\max \{|b_{2s+1}|,|b_{2s+2}|\} = \max \{\frac{1}{k}|b_{s+1}|,|b_{2s+1}|\} = b_{2s+1}>0$.  
So in this case our inequality is also proved.

Now let $a \in \mathbb{Q}$ be a root of the Stern polynomial $B_n$. If $a \neq 0$ then we can assume that $n$ is odd. Moreover $B_{2s+1}( 0) =1$
for each integer $s$, so $a$ is of the form $\pm \frac{1}{k}$ for some integer $k$.  But 
all coefficients of Stern polynomials are non-negative, so $a = -\frac{1}{k}$. Now our theorem follows from inequality \eqref{eq1} .  

\end{proof}

Now for each $a \in \{0,-1,-\frac{1}{2},-\frac{1}{3} \}$ we want to characterize those $n$ for which $B_n( a) =0$. Let us denote
$R_a=\{n \in \mathbb{N} | B_n( a)  = 0\}$. One can check that $B_n( 0)  = n \pmod 2$ and
$B_n( -1)  = ( n+1)  \pmod 3 -1$, so in this case our problem is easy and we get $R_0=2\mathbb{N}$ and  $R_{-1}=3\mathbb{N}$. 
The problem is much more complicated when $a \in \{-\frac{1}{2},-\frac{1}{3}\}$. 
\begin{thm} The sets $R_{- \frac 12}, R_{- \frac 13} $ satisfies the following properties
 \begin{enumerate}[(a)]  
	     \item $2R_{ - \frac 12 }\subset R_{ - \frac 12 },2R_{ - \frac 13} \subset R_{ - \frac 13}$,
             \item $R_{-\frac{1}{2}} \subset 5\mathbb{N}$, $R_{-\frac{1}{3}} \subset 21 \mathbb{N}$,
             \item If $5n \in R_{-\frac{1}{2}}$, and $2^k>5n$ then $5( 2^k \pm n)  \in R_{-\frac{1}{2}}$. 
	     \item If $21n \in R_{-\frac{1}{3}}$, and $2^k >21n$ then $21( 2^k \pm n)  \in R_{-\frac{1}{3}}$. 
 \end{enumerate}
\begin{proof}
 Part $( a) $ is obvious. We can look only for odd elements of $R_{-\frac 12},R_{- \frac 13}$. Looking at the
reductions modulo $5$ we get $B_n( - \frac{1}{2})  \equiv B_n( 2)  = n \pmod 5$ so $R_{- \frac 12} \subset 5 \mathbb{N}$.
By looking modulo $14$ we get $R_{- \frac 13} \subset 21 \mathbb{N}$. 
Let us prove part $( c) $. We prove by induction on $n$ that $B_{5 \cdot 2^k \pm n}( - \frac 12)  = \mp \frac 14
B_{n}( - \frac 12) $, when $2^k \geq n$. For $n=1$ then we have to prove $B_{5 \cdot 2^k \pm 1}( - \frac 12)  = \mp \frac 14$,
which can be easily proven by induction on $k$. Now if $n=2s$ then
\begin{equation*}
B_{5 \cdot 2^k \pm 2s}\l( - \frac 12 \r)  = - \frac 12 B_{5 \cdot 2^{k-1} \pm s}\l( - \frac 12\r)  = 
- \frac 12 \l(  \mp \frac 14 B_s\l( - \frac 12\r) \r)  = \mp \frac 14 B_{2s}\l( - \frac 12\r) .    
\end{equation*}
If $n=2s+1$ then
\begin{align*}
B_{5 \cdot 2^k \pm ( 2s+1 ) }\l( - \frac 12\r)  &= B_{5 \cdot 2^{k-1} \pm s}\l( - \frac 12\r)  + 
B_{5 \cdot 2^{k-1} \pm ( s+1) }\l( - \frac 12\r) \\ &= \mp \frac 14 \l( B_s \l( - \frac 12\r) +B_{s+1}\l( - \frac 12\r)  \r) 
= \mp \frac 14 B_{2s+1}\l( -\frac 12\r) .  
\end{align*}
Which completes the induction step, and part $( c) $ follows. In the same manner one can prove that for $2^k>n$ we have
$B_{21 \cdot 2^k \pm n}( - \frac 13)  = \mp \frac 1{27} B_{n}( -\frac 13) $ which implies part $( d) $.   
\end{proof}
\end{thm}

Using previous theorem we can check that no other reductions of our sequences are eventually periodic. 

\begin{col} \begin{enumerate}[(a)]
             \item For each prime number $p \neq 2,5$ the sequence $( B_n( - \frac 12)  \pmod p) _{n \geq 0}$ is not eventually periodic.
	     \item For each prime number $p \neq 2,3,7$ the sequence $( B_n( - \frac 13)  \pmod p) _{n \geq 0}$ is not eventually periodic.
            \end{enumerate}
\end{col}
\begin{proof}
Let us prove part $( a) $. Suppose that for $n>N$ we have $B_n( -\frac 12)  \equiv B_{n+2^u( 2v+1) }( - \frac 12)  \pmod p$.
Let us take $n$ such that $B_{5n}( - \frac 12)  \equiv 0 \pmod p$. 
We choose $k$ such that $2^{k+u}-1 = ( 2v+1) s$ for some integer $s$, 
and $2^{k+u}>5\cdot 2^u n$. From previous theorem we have that $B_{5( 2^u n+2^{k+u}) }( - \frac 12)  \equiv 0 \pmod p$, 
so $0 \equiv B_{5( 2^u n+2^{k+u}) -5( 2^u( 2v+1) s) }( - \frac 12)   \equiv B_{5( 2^u n+2^u) }( - \frac 12)  \equiv B_{5( n+1) }( - \frac 12) 
\pmod p$,
hence all numbers $M$ divisible by $5$ greater than $N$ satisfy $B_M( - \frac 12)  \equiv 0 \pmod p$. But one can check that
$B_{15+5\cdot 2^{n}} = -\frac 5{32}$ for $n > 3$ and this number is not $0 \pmod p$. We get a contradiction. 
Proof of part $( b) $ can be done in the same manner. 
\end{proof}

Using reductions modulo prime numbers we prove that the sets $R_{- \frac 12},R_{- \frac 13}$ has lower density $0$.

\begin{lem}
 Let $p>3$ be a prime number, $t$ be a fixed constant from the set $\{- \frac 12,- \frac 13\}$, and $d_{n,p} = \frac{ \# \{ 0 \leq k < n | B_k( t)  \equiv 0 \pmod p \} }{n}$.
 We have that 
\begin{equation*}
 \lim_{n \to \infty} \frac{d_{2^0,p}+d_{2^1,p}+\ldots+d_{2^n,p}}{n+1} \leq \frac{2}{\log p}.
\end{equation*}
\end{lem}
\begin{proof} We denote by $b_n$ the sequence $B_n( t)  \pmod p$. 
 Let us consider two operators $E_0,E_1 : (  \mathbb N \to \mathbb{F}_p)  \to (  \mathbb N \to \mathbb{F}_p) $ defined
by formulas $E_0( ( a_n) _{n \geq 0})  = ( a_{2n}) _{n \geq 0}$, and $E_1( ( a_n) _{n \geq 0})  = ( a_{2n+1}) _{n \geq 0}$.
Let us observe that the set $\{( \alpha b_n + \beta b_{n+1}) _{n \geq 0} | \alpha,\beta \in \mathbb{F}_p\}$ is closed under
action of $E_0,E_1$ because
\begin{equation*}
  E_0(  ( \alpha b_n + \beta b_{n+1}) )  =( \alpha b_{2n} + \beta b_{2n+1}) =
( ( t\alpha+\beta) b_n + \beta b_{n+1}) 
\end{equation*}
and 
\begin{equation*}
 E_1( ( \alpha b_n + \beta b_{n+1}) )  = ( \alpha b_{2n+1} + \beta b_{2n+2})  = ( \alpha b_{n} + ( \alpha+t\beta) b_{n+1}) .  
\end{equation*}
Therefore $( b_n) $ has a finite orbit under actions of $E_0,E_1$ so our sequence is $2$-automatic in sense of \cite{ALU}. We construct automaton with states labeled with $( \alpha,\beta) $ for each $\alpha,\beta \in \mathbb{F}_p$,
and edges $( \alpha,\beta)  \overset{0}{\rightarrow} ( t\alpha+\beta,\beta) $ and $( \alpha,\beta)  \overset{1}{\rightarrow} ( \alpha,t\beta+\alpha) $. The initial state is $( 1,0) $. For each state $( \alpha,\beta) $
we define its projection $\pi( \alpha,\beta)  = \beta$. One can observe that if $\bin( k)  = w_0w_1\ldots w_s$ then
$b_k =( ( E_{w_0} \circ E_{w_1} \circ \ldots \circ E_{w_s}) (  ( b_n)  ) ) _0$, so to compute $b_k$ we start at initial state
then go through the path $w_s w_{s-1} \ldots w_0$ and finally use projection function. 

We call a state reachable if it is reachable from initial state. Let us observe that 
the state $( 0,1) $ is reachable. We have
\begin{equation*}
 ( 1,0)  \overset{1}{\rightarrow} ( 1,1)  \overset{0}{\rightarrow} ( t+1,1)  \overset{0}{\rightarrow} ( t^2+t+1,1)  
\overset{0}{\rightarrow} \ldots \overset{0}{\rightarrow} ( t^{p-2}+\ldots+t+1,1) =( 0,1) . 
\end{equation*}
Now we can see that the state $( \alpha,\beta) $ is reachable if and only if the state $( \beta,\alpha) $
is reachable ( just go to $( 0,1) $ and then use opposite arrows ) . Now observe that if for some $u \in \mathbb{F}_p^*$
the state $( u,0) $ is reachable then the following different states are also reachable
\begin{equation*}
 ( 0,u)  \overset{0}{\rightarrow} ( u,u)  \overset{0}{\rightarrow} ( u( t+1) ,u)  \overset{0}{\rightarrow} 
( u( t^2+t+1) ,u) )  \overset{0}{\rightarrow} \ldots \overset{0}{\rightarrow} ( u( t^{\ord( t) -1}+\ldots+t+1) ,u) 
\end{equation*}
 Where $\ord( t) $ is the order of $t$ in group $\mathbb{F}_p^*$ and of course $\ord( t)  \geq \log_3 p -1  \geq \frac 12 \log p -1$. Hence
at most $\frac{2}{\log p}$ of reachable states are of the form $( u,0) $. Moreover from each reachable state there is a path to initial state.
Indeed, from state $( \alpha,\beta) $ we can go to $\begin{pmatrix} t & 1 \\ 0 & 1 \end{pmatrix} \begin{pmatrix} \alpha \\ \beta \end{pmatrix}$
or to $\begin{pmatrix} 1 & 0 \\ 1 & t \end{pmatrix} \begin{pmatrix} \alpha \\ \beta \end{pmatrix}$, those matrices has finite order in $GL_2( \mathbb{F}_p) $
so we can go backwards.

Let $M$ be adjacency matrix of strongly connected component of our automaton as a directed graph containing initial
vertex. Let $A=\frac 12 M$. Using a straightforward application of Frobenius-Perron theorem 
( see for example \cite[example 8.3.2 ,p.677]{Meyer})  we get that there is a limit
\begin{equation}\label{eqfrob}
 \lim_{k \to \infty} \frac{I+A+A^2+\ldots+A^k}{k+1} = G. 
\end{equation}
Of course $AG=GA=G$. As every vertex of our strongly connected component has inner and outer degree $2$, one can observe
that each element of the matrix $G$, namely $g_{s,t}$, is an arithmetic mean of $g_{s,t_0}$ and $g_{s,t_1}$ where 
$t \overset{0}{\rightarrow} t_0$ and $ t \overset{1}{\rightarrow} t_1$, and on the other hand an arithmetic mean of 
$g_{s_0,t}$ and $g_{s_1,t}$ where $s_0 \overset{0}{\rightarrow} s$ and $s_1 \overset{1}{\rightarrow} s$. Hence using 
strongly connectedness we get that 
\begin{equation*} G = \frac 1K   \begin{pmatrix} 1 & 1 & \cdots & 1 \\ 1 & 1 & \cdots & 1 \\
                                      \vdots & \vdots & \ddots & \vdots \\
					1 & 1 & \cdots & 1  
                                     \end{pmatrix},
\end{equation*}
where $K$ is the cardinality of our component. 

Number $d_{2^n,p}$ is the number of paths of length $n$ from initial state to one of states of the form
$( u,0) $ divided by $2^n$. We know that $A^n_{s_1,s_2}$ is the number of paths of length $n$ from $s_1$ to $s_2$,
divided by $2^n$. Therefore from equation  \eqref{eqfrob} we get that
\begin{equation*}
  \lim_{n \to \infty} \frac{d_{0,p}+d_{2^0,p}+d_{2^1,p}+\ldots+d_{2^n,p}}{n+1} = \frac{ \#\{ \text{ reachable states of the form } ( u,0) \} }{K} \leq \frac{2}{\log p}.
\end{equation*}
 
Our lemma is proved.  
\end{proof}
The above result implies the following

\begin{col} 
The sets $R_{- \frac 12}, R_{- \frac 13}$ has lower density $0$. 
\end{col}
\begin{proof}
 Let $t$ be a fixed constant from the set $\{- \frac 12 , - \frac 13\}$. Let 
$d_{n} = \frac{ \# \{ 0 \leq k < n | B_k( t)  = 0 \} }{n}$. Of course $d_{n,p} \geq d_n$ for each $p$. 
From our lemma we get 
\begin{equation*}
\limsup_{n \to \infty} \frac{d_{2^0}+d_{2^1}+\ldots+d_{2^n}}{n+1} \leq  \frac{2}{\log p},
\end{equation*}
therefore \begin{equation*} \lim_{n \to \infty} \frac{d_{2^0}+d_{2^1}+\ldots+d_{2^n}}{n+1} = 0,\end{equation*}
and finally \begin{equation*} \liminf_{n \to \infty} d_n = 0. \end{equation*}  
\end{proof}

We expect that $R_{- \frac 12}, R_{- \frac 13}$ has upper density $0$. 
\begin{conj}
 The sets $R_{- \frac 12}, R_{- \frac 13}$ has density $0$. 
\end{conj}

\end{section}

\begin{section}{Consecutive polynomials with equal degrees}
In this section we will characterize those $n$ for which $e( n) =e( n+1) $. Resolving conjecture posted in \cite[Conjecture 6.3]{UL1}. 
We also characterize those $n$ for which $e( n) =e( n+1) =e( n+2) $. We recall the recurrence satisfied be the
sequence $( e_n) _{n \in \mathbb{N}}$ of degrees of Stern Polynomials   
\begin{equation*}
\begin{cases}
  e( 0) =1 \\
  e( 2n)  = e( n) +1 \\
  e( 4n+1)  = e( n) +1 \\
  e( 4n+3)  = e( n+1) +1 \\
 \end{cases}.
\end{equation*} 
Let us the define sequences $( p_n) _{n \mathbb{N}},( q_n) _{n \mathbb{N}}$ as follows 
\begin{equation*}
 p_n = \frac{4^n-1}{3}, q_n = \frac{5 \cdot 4^n -2}{3}, n \geq 0.
\end{equation*}

\begin{thm} We have the following equality
\begin{equation*}
\{n | e( n) =e( n+1) \} = \{2^{2k+1}n +p_k | k,n \in \mathbb{N}_+ \} \cup \{ 2p_k | k \in \mathbb{N}_+ \} \cup 
\{ 2^{2k+1}n + q_k | k,n \in \mathbb{N}, k>0 \}.
\end{equation*} 
 
\end{thm}
\begin{proof} We will use binary representations of integers. We write $\bin( m) $ for sequence of digits  of $m$ in base two.
Let us observe that $\bin( p_k)  = 1( 01) ^{k-1},$ and $ \bin( q_k)  = 1( 10) ^k$ for each $k>0$. We consider eight cases 
depending on the binary expansion of $n$.
\begin{enumerate}[(i)]
 \item If $n = 4k$ then we have
 \begin{equation*}
  e( n)  = e( 4k)  = e( 2k) +1 = e( k) +2 
 \end{equation*}
 and 
 \begin{equation*}
  e( n+1)  = e( 4k+1)  = e( k) +1.
 \end{equation*}
 In particular $e( n)  \neq e( n+1) $ in this case.

 \item If $n = 4k+3$ then we have
  \begin{equation*}
e( n) =e( 4k+3) =e( k+1) +1
 \end{equation*}

 and 
\begin{equation*}
e( n+1) =e( 4k+4) =e( k+1) +2.
 \end{equation*}
We thus get $e( n)  \neq e( n+1) $.
 \item If $n=\overline{\bin( m) 0( 01) ^k}$, where $m,k \in \mathbb{N}_+$ then we have
\begin{equation*}
e( n)  = e( \overline{\bin( m) 0( 01) ^k})  = e( \overline{\bin( m) 0}) +k=e( m) +k+1
\end{equation*}
 and 
 \begin{align*}
  e( n+1)  &= e( \overline{\bin( m) 0( 01) ^{k-1}10})  = e( \overline{\bin( m) 00( 10) ^{k-2}11}) +1 \\
&= e( \overline{\bin( m) 00( 10) ^{k-3}11}) +2 = \ldots = e( \overline{\bin( m) 01}) +k =e( m) +k+1. \\ 
 \end{align*}
Our computation implies $e( n)  = e( n+1) $.

\item If $n= \overline{\bin( m) 11( 01) ^k}$, where $k,m \in \mathbb{N}$, $k>0$, then we have
\begin{equation*}
e( n)  = e( \overline{\bin( m) 11( 01) ^k}) =k+e( \overline{\bin( m) 11}) =k+1+e( \overline{\bin( m+1) })  
\end{equation*}
and
\begin{align*}
e( n+1)  &= e( \overline{\bin( m) 11( 01) ^{k-1}10}) =e( \overline{\bin( m) 1( 10) ^{k-1}11}) +1=e( \overline{\bin( m) 1( 10) ^{k-2}11}) +2 \\
&= e( \overline{\bin( m) 111}) +k = e( \overline{\bin( m+1) 0}) +k+1=e( \overline{\bin( m+1) }) +k+2. \\
\end{align*}

We get $e( n)  \neq e( n+1) $ in this case.

 \item If $n=\overline{( 01) ^k}$, where $k \in \mathbb{N}_+$, then we have
\begin{equation*}e( n)  = e( \overline{( 01) ^k})  = e( \overline{1}) +k-1=k-1 \end{equation*}
and
\begin{align*}
 e( n+1)  &= e( \overline{( 01) ^{k-1}10})  = e( \overline{( 01) ^{k-1}1}) +1=e( \overline{( 01) ^{k-2}011}) +1 \\
&= e( \overline{( 01) ^{k-2}1}) +2) =e( \overline{11}) +k-1= k, \\   
\end{align*}
and thus $e( n)  \neq e( n+1) $.

\item If $n = \overline{\bin( m) 00( 10) ^k}$, where $m,k \in \mathbb{N}_+$, then we have
\begin{equation*}
   e( n)  = e( \overline{\bin( m) 00( 10) ^k})  = k+2+e( \overline{\bin( m) })  
\end{equation*}
and 
\begin{equation*}
 e( n+1)  = e( \overline{\bin( m) 0( 01) ^k1})  = e( \overline{\bin( m) 0( 01) ^{k-1}1}) +1 = e( \overline{\bin( m) 01}) +k=
 e( \overline{\bin( m) }) +k+1.
\end{equation*}
Similarly as in the previous case we get $e( n)  \neq e( n+1) $.

\item If $n = \overline{( 10) ^k}$ where $k \in \mathbb{N}_+$, then we have
\begin{equation*}
e( n) =e( \overline{( 10) ^k}) =e( \overline{( 01) ^{k}}) +1=e( 1) +k=k  
\end{equation*}
and
\begin{equation*}
    e( n+1)  = e( \overline{( 10) ^{k-1}11})  = e( \overline{( 10) ^{k-2}11}) +1=e( \overline{11}) +k-1=k, 
\end{equation*}
and thus $e( n)  = e( n+1) $.

\item If $n = \overline{\bin( m) 1( 10) ^k}$ where $m,k \in \mathbb{N}$, $k>0$, then we have 
\begin{equation*}
e( n)  = e( \overline{\bin( m) 1( 10) ^k})  = e( \overline{\bin( m) 11( 01) ^{k-1}}) +1 = e( \overline{\bin( m+1) }) +k+1 
\end{equation*}

and

\begin{align*}
  e( n+1)  &= e( \overline{\bin( m) 1( 10) ^{k-1}11})  = e( \overline{\bin( m) 1( 10) ^{k-2}11}) +1  \\
&= e( \overline{\bin( m) 111}) +k-1 = e( \overline{\bin( m+1) }) +k+1. \\
\end{align*}
and thus $e( n)  = e( n+1) $. 

\end{enumerate}

All numbers of the form $4k+1$ can be written uniquely in one of the forms $( iii) ,( iv) ,( v) $, and all numbers of the
form $4k+2$ can be can be written uniquely in one of the forms $( vi) ,( vii) ,( viii) $. Therefore all cases were considered. In cases $( iii) ,( vii) ,( viii) $ we have that $e( n) =e( n+1) $, and one can observe that
the numbers in this cases are exactly the elements of the set $\{2^{2k+1}n +p_k | k,n \in \mathbb{N}_+ \} \cup \{ 2p_k | k \in \mathbb{N}_+ \} \cup 
\{ 2^{2k+1}n + q_k | k,n \in \mathbb{N}, k>0 \}$. So our theorem is proved. 
\end{proof}

In the next theorem we characterize the set of $n$ for which $e( n)  = e( n+1)  = e( n+2) $.  

\begin{thm} We have the following equality
 \begin{multline*}
\{n | e( n) =e( n+1) =e( n+2)  \} = \\ = \{2^{2k+1}n +p_k | k,n \in \mathbb{N}_+,k \geq 2 \} \cup \{ 2p_k-1 | k \geq 2 \} \cup 
\{ 2^{2k+1}n + q_k-1 | k,n \in \mathbb{N}, k \geq 2 \}.
\end{multline*} 
\end{thm}
\begin{proof} 
Let $A= \{2^{2k+1}n +p_k | k,n \in \mathbb{N}_+ \} $, $B = \{ 2p_k | k \in \mathbb{N}_+ \}$, and
$C = \{ 2^{2k+1}n + q_k | k,n \in \mathbb{N}, k>0 \}$. We want to compute $( A \cup B \cup C)  \cap ( ( A \cup B \cup C) +1) $. 
But as every element of $A$ equals $1 \pmod 4$ and every element of $B \cup C$ equals $2 \pmod 4$, it is enough to compute
$( ( A+1)  \cap B)  \cup ( ( A+1)  \cap C) $. First let us compute $( A+1)  \cap B$, we have
\begin{equation*}
 2^{2k+1}n+p_k+1 = 2p_l, \text{ for } n,k,l \geq 1.
\end{equation*}
By a simple computations we get
\begin{equation*}
 3 \cdot 2^{2k+1}n+4^k = 2 \cdot 4^l-4.
\end{equation*}
Looking modulo $ 4$ we get that $k=1$ or $l=1$. When $k=1$ then our equality became $12n+4 = 4^l$. For each $l$ we have exactly one $n$
such that this equality holds, but as $n \geq 2$ we have $l \geq 2$. 
If $l=1$ then our equality became $2^{2k+1}n+p_k+1=2$ which contradicts $n,k \geq 1$. Summarizing $( A+1)  \cap B = \{ 2p_l-1 | l \geq 2 \}$.
Let us compute $( A+1)  \cap C$. We have
\begin{equation*}
 2^{2k+1}n+p_k+1 =2^{2l+1}m+q_l, \text{ for } n,k,l \geq 1, m \geq 0. 
\end{equation*}
By a simple computations we get
\begin{equation*}
2^{2k+1} \cdot 3n+4^k = 2^{2l+1} \cdot 3m + 5 \cdot 4^l-4.
\end{equation*}
Looking $\pmod 4$ we get that $k=1$ or $l=1$. In the first case we get $8n+2 = 2^{2l+1}m+q_l$. For each $l,m$ such that $l \geq 2$ we have
exactly one $n$ such that this equality holds. So we get $\{ 2^{2l+1}m + q_l-1 | m,l \in \mathbb{N}, l \geq 2 \}$.
In the second case we get  $\{2^{2k+1}n +p_k | k,n \in \mathbb{N}_+,k \geq 2 \}$ in the same way. Summarizing we get
$( A+1)  \cap C =  \{ 2^{2l+1}m + q_l-1 | m,l \in \mathbb{N}, l \geq 2 \} \cup \{2^{2k+1}n +p_k | k,n \in \mathbb{N}_+,k \geq 2 \}$.
And our result follows.
\end{proof}
\end{section}

\begin{section}{Reciprocal Stern polynomials}

Let us recall that the polynomial $B_n( t) $ is called reciprocal if $B_n( t) =t^{e( n) }B_n( \frac 1t) $. 
It was observed in \cite{UL1} that for each $n$ polynomials $B_{2^n-1}( t) $ and $B_{2^n-5}( t) $ are reciprocal.  
We define sequences $( u_n) _{n \in \mathbb{N}}$ and $( v_n) _{n \in \mathbb{N}}$ as follows
\begin{equation*}
\begin{cases} 
 u_0 = 1 & u_n = 2^{8n-2} u_{n-1}+2^{4n}-1, \\ 
v_0 = 5 & v_n = 2^{8n+2}v_{n-1}-2^{4n+2}+1,
 \end{cases}
\end{equation*}
and show that whenever $2^k>u_n$ (or $2^k>v_n$)  then $B_{2^k-u_n}$ (or $B_{2^k-v_n}$)  is reciprocal.

\begin{thm} 
For each integer $n \geq 0 $ all polynomials in the sequences $( B_{2^k-v_n}) _{k \geq 4n^2+6n+3}$, and 
$( B_{2^k-u_n}) _{k \geq 4n^2+2n+1}$ are reciprocal. 
\end{thm}
\begin{proof}
Let us observe that $\bin( u_n)  = 10^21^40^6\ldots 0^{4n-2}1^{4n}$ and $\bin( v_n)  = 10^21^4\ldots 1^{4n}0^{4n+1}1$.
Using recursive formula for $e( n) $, binary representations and induction we will get that 
$e( 2^k-u_n) =k-2n-1$ and $e( 2^k-v_n) =k-2n-2$. 
 Let us prove our theorem. We proceed by induction. 
 For $u_0=1$ we have $B_{2^k-1} = t^{k-1}+t^{k-2}+\ldots+t+1$ and our conditions are satisfied. For $v_0=5$ we have 
 \begin{align*}
 B_{2^k-5}( t) &=( 1+t) B_{2^{k-2}-1}( t) +tB_{2^{k-3}-1}( t) \\
 &=( 1+t) t^{k-3}B_{2^{k-2}-1}( t^{-1}) +t^{k-3}B_{2^{k-3}-1}( t^{-1})  \\
 &=t^{k-2} \left(  ( 1+t^{-1}) B_{2^{k-2}-1}( t^{-1}) +t^{-1}B_{2^{k-3}-1}( t^{-1})  \right)  \\
 &=t^{k-2}B_{2^k-5}( t^{-1}) . \\ 
 \end{align*}
 Let us compute
 \begin{align*}
  B_{2^k-u_n}( t)  &= B_{2^k-\overline{10^21^40^6\ldots 0^{4n-2}1^{4n}}}( t)   \\
  &= B_{2^{k-1}-\overline{10^21^40^6\ldots 0^{4n-2}1^{4n-1}}}( t) +B_{2^{k-1}-\overline{10^21^40^6\ldots 0^{4n-3}10^{4n-1}}}( t)   \\
  &= B_{2^{k-1}-\overline{10^21^40^6\ldots 0^{4n-2}1^{4n-1}}}( t) +t^{4n-1}B_{2^{k-4n}-\overline{10^21^40^6\ldots 0^{4n-3}1}}( t)   \\
  &=B_{2^{k-1}-\overline{10^21^40^6\ldots 0^{4n-2}1^{4n-1}}}( t) +t^{4n-1}B_{2^{k-4n}-v_{n-1}}( t)  \\
  &= \ldots \\
  &= B_{2^{k-4n}-\overline{10^21^40^6\ldots 0^{4n-2}}}( t) +( t^{4n-1}+t^{4n-2}+\ldots+t+1) B_{2^{k-4n}-v_{n-1}}( t)   \\
  &= t^{4n-2}B_{2^{k-8n+2}-u_{n-1}}( t) +( t^{4n-1}+t^{4n-2}+\ldots+t+1) B_{2^{k-4n}-v_{n-1}}( t) . = ( *)  \\  
  \end{align*}
Using induction hypothesis we get
 \begin{align*}
( *)  &= t^{k-6n+1}B_{2^{k-8n+2}-u_{n-1}}( t^{-1}) + 
( t^{4n-1}+t^{4n-2}+\ldots+t+1) t^{k-6n}B_{2^{k-4n}-v_{n-1}}( t^{-1})  \\
&= t^{k-2n-1}\left( t^{2-4n}B_{2^{k-8n+2}-u_{n-1}}( t^{-1}) +( t^{4n-1}+t^{4n-2}+\ldots+t+1) t^{1-4n}B_{2^{k-4n}-v_{n-1}}( t^{-1})  \right)  \\
&= t^{k-2n-1}B_{2^k-u_n}( t^{-1}) .\\
\end{align*} 
So $B_{2^k-u_n}$ is reciprocal. In the same way one can prove that $B_{2^k-v_n}$ is reciprocal. We omit the details. 
\end{proof}

\begin{col} Let $Rec = \{n | B_n( t) =t^{e( n) }B_n( t^{-1}) \}$. We have that $\#( Rec \cap [1,n])  = \Omega(  \log( n) ^{3/2}) $.  
\end{col}
\begin{proof}
 Let $n>2^k$ we have that 
\begin{align*}
\#( Rec \cap [1,n])  & \geq \#\{ ( i,m)  | u_m < 2^i , i < k\} = \sum_{u_m<2^k} ( k-\log( u_m) )   \\
  &= \sum_{4m^2+2m+1<k} ( k-( 4m^2+2m+1) )  = \frac 12 k\sqrt{k} - \frac 43 \left( \frac{\sqrt{k}}{2} \right) ^3 + o( k\sqrt{k})  = \theta( k \sqrt{k}) , 
\end{align*}
and our result follows.
 
\end{proof}

We expect the following conjecture about density of the set $Rec$
\begin{conj}
 The function $f( n)  = \# ( Rec \cap [1,n]) $  is $O( \log( n) ^k) $ for 
some constant $k$. 
\end{conj}

\end{section}

\footnote{Maciej Gawron}
\footnote{Jagiellonian University}
\footnote{Institute of Mathematics}  
\footnote{\L{}ojasiewicza 6, 30-348 Krak\'{o}w, Poland}
\footnote{e-mail: \texttt{maciej.gawron@uj.edu.pl} }

\end{document}